\documentclass{birkjour}
\usepackage{hyperref}
\newtheorem{theorem}{Theorem}[section]
\newtheorem{corrolary}[theorem]{Corollary}
\newtheorem{lemma}[theorem]{Lemma}
\newtheorem{prop}[theorem]{Proposition}
\theoremstyle{definition}

\theoremstyle{remark}
\newtheorem{remark}[theorem]{Remark}

\DeclareMathOperator{\id}{id}
\DeclareMathOperator{\sign}{sgn}

\renewcommand{\theenumi}{\roman{enumi}}
\renewcommand{\labelenumi}{(\theenumi)}
\newcommand{\clopen}[2]{\char"5B #1,#2\char"29}

\begin{document}

\title[On a problem of J.~Matkowski and J.~Weso\l owski]{On a problem of Janusz Matkowski and Jacek Weso\l owski}

\author[J.~Morawiec]{Janusz Morawiec}
\address{Instytut Matematyki\\
Uniwersytet \'{S}l\c aski\\
Bankowa 14\\
PL-40-007 Katowice\\
Poland}
\email{morawiec@math.us.edu.pl}

\author[T.~Z\"{u}rcher]{Thomas Z\"{u}rcher}
\address{Mathematics Institute\\ University of Warwick\\
Coventry\\
CV4 7AL\\
UK}
\email{T.Zurcher@warwick.ac.uk}


\begin{abstract}
We study the problem of the existence of increasing and continuous solutions $\varphi\colon[0,1]\to[0,1]$ such that $\varphi(0)=0$ and $\varphi(1)=1$ of the functional equation
\begin{equation*}
\varphi(x)=\sum_{n=0}^{N}\varphi(f_n(x))-\sum_{n=1}^{N}\varphi(f_n(0)),
\end{equation*}
where $N\in\mathbb N$ and $f_0,\ldots,f_N\colon[0,1]\to[0,1]$ are strictly increasing contractions satisfying the following condition $0=f_0(0)<f_0(1)=f_1(0)<\cdots<f_{N-1}(1)=f_N(0)<f_N(1)=1$. In particular, we give an answer to the problem posed in \cite{M1985} by Janusz Matkowski concerning a very special  case of that equation.
\end{abstract}

\maketitle


\section{Introduction}
During the 47th International Symposium on Functional Equations in 2009 Jacek Weso\l owski asked whether the identity on $[0,1]$ is the only increasing and continuous solution $\varphi\colon [0,1]\to [0,1]$ of the equation
\begin{equation*}\label{e0}
\varphi(x)=\varphi\left(\frac{x}{2}\right)+\varphi\left(\frac{x+1}{2}\right)
-\varphi\left(\frac{1}{2}\right)\tag{$\textsf{e}_1$}
\end{equation*}
satisfying
\begin{equation}\label{cond}
\varphi(0)=0\hspace{3ex}\hbox{ and }\hspace{3ex}\varphi(1)=1.
\end{equation}
This question has been posed in connection with studying probability measures in the plane which are invariant by \lq\lq winding\rq\rq\ (see \cite{MW2012}).

A negative answer to this question has been obtained in \cite{2016} and reads as follows.

\begin{theorem}\phantomsection\label{thm11}
\begin{enumerate}
\item The identity on $[0,1]$ is the only increasing and absolutely continuous solution $\varphi\colon [0,1]\to [0,1]$ of equation~\eqref{e0} satisfying~\eqref{cond}.
\item\label{representation} For every $p\in(0,1)$ the function $\varphi_p\colon[0,1]\to[0,1]$ given by
\begin{equation}\label{p}
\varphi_p\left(\sum_{k=1}^\infty\frac{x_k}{2^k}\right)=
\sum_{k=1}^{\infty}x_k p^{k-\sum_{i=1}^{k-1}x_i}(1-p)^{\sum_{i=1}^{k-1}x_i},
\end{equation}
where $x_k\in\{0,1\}$ for all $k\in\mathbb N$, is an increasing and continuous solution of equation~\eqref{e0} satisfying~\eqref{cond}. Moreover, $\varphi_p$ is singular for every $p\neq\frac{1}{2}$.
\end{enumerate}
\end{theorem}

Let us note that the first assertion of Theorem~\ref{thm11} is known (see e.g.\ \cite{R1957} or \cite{LP1977}), however in \cite{2016} we can find an independent proof of it.

It turns out that in 1985 Janusz Matkowski posed a problem asking if equation~\eqref{e0} has a non-linear monotonic and continuous solution $\varphi\colon[0,1]\to\mathbb R$ (see \cite{M1985}). Moreover, he observed that monotonic solutions of equation~\eqref{e0} are connected with invariant measures for a certain map on $[0,1]$. Note that Matkowski's problem is equivalent to Weso\l owski's question.
\begin{remark}\phantomsection\label{rem12}
\begin{enumerate}
\item If $\varphi\colon[0,1]\to[0,1]$ is an increasing and continuous solution of equation~\eqref{e0} satisfying~\eqref{cond}, then for all $a,b\in\mathbb R$ the function $a\varphi+b$ is monotonic, continuous and satisfies~\eqref{e0} for every $x\in[0,1]$.
\item If $\varphi\colon[0,1]\to\mathbb R$ is a monotonic and continuous  solution of equation~\eqref{e0}, different from a constant function, then $\frac{\varphi-\varphi(0)}{\varphi(1)-\varphi(0)}$ is an increasing and continuous function satisfying \eqref{cond} and~\eqref{e0} for every $x\in[0,1]$.
\end{enumerate}
\end{remark}


\section{Preliminaries}\label{Preliminaries}
Fix $N\in\mathbb N$, strictly increasing contractions $f_0,\ldots,f_N\colon[0,1]\to[0,1]$ such that
\begin{equation}\label{0<1}
0=f_0(0)<f_0(1)=f_1(0)<\cdots<f_{N-1}(1)=f_N(0)<f_N(1)=1
\end{equation}
and consider the functional equation
\begin{equation*}\label{e}
\varphi(x)=\sum_{n=0}^{N}\varphi(f_n(x))-\sum_{n=1}^{N}\varphi(f_n(0))
\tag{\textsf{E}}
\end{equation*}
for every $x\in[0,1]$. Denote by $\mathcal C$ the class of all continuous and increasing solutions $\varphi\colon[0,1]\to[0,1]$ of equation~\eqref{e} satisfying~\eqref{cond}. Following the idea from \cite{2016} we show that $\mathcal C$ contains many functions, however, we manage to identify a quite large class of contractions that includes the similitudes such that there is exactly one absolutely continuous solution.

We begin with two observations showing that in many situations the class $\mathcal C$ is determined by two of its subclasses $\mathcal C_a$ and $\mathcal C_s$, consisting of all absolutely continuous and all singular functions, respectively.

\begin{remark}\label{rem21}
If $\varphi_1,\varphi_2\in\mathcal C$ and if $\alpha\in(0,1)$, then
$\alpha \varphi_1+(1-\alpha)\varphi_2\in\mathcal C$.
\end{remark}

To formulate the next remark we recall that a Lebesgue measurable function $f\colon [0,1]\to [0,1]$ is said to be \emph{nonsingular} if the set $f^{-1}(A)$ has Lebesgue measure zero for every set $A\subset [0,1]$ of Lebesgue measure zero (see \cite{LM1994}). Observe that an invertible Lebesgue measurable function $f$ is nonsingular if and only if its inverse $f^{-1}$ satisfies Luzin's condition~(N).

\begin{remark}\label{rem22}
Assume that all the contractions $f_0,\ldots, f_N$ are nonsingular. Then, both the absolutely continuous and the singular parts of every element from $\mathcal C$ satisfy $\eqref{e}$ for every $x\in[0,1]$. 	 	
\end{remark}

\begin{proof}
Fix $\varphi\in\mathcal C$ and denote by $\varphi_a$ and $\varphi_s$ its absolutely continuous and singular parts\footnote{The parts are unique up to a constant. For definiteness, we choose them such that $\varphi_a(0)=\varphi_s(0)=0$.}, respectively. By~\eqref{e}, for every $x\in[0,1]$ we have
\begin{equation*}
\varphi_a(x)-\sum_{n=0}^{N}\varphi_a(f_n(x))=
-\varphi_s(x)+\sum_{n=0}^{N}\varphi_s(f_n(x))-\sum_{n=1}^{N}\varphi(f_n(0)),
\end{equation*}
and hence there exists a real constant $c$ such that
\begin{equation*}
\varphi_a(x)-\sum_{n=0}^{N}\varphi_a(f_n(x))=c
\quad\hbox{and}\quad
-\varphi_s(x)+\sum_{n=0}^{N}\varphi_s(f_n(x))-\sum_{n=1}^{N}\varphi(f_n(0))=c.
\end{equation*}
This jointly with the fact $f_0(0)=0$ stipulated in~\eqref{0<1} gives
\begin{equation*}
c=\varphi_a(0)-\sum_{n=0}^{N}\varphi_a(f_n(0))=-\sum_{n=1}^{N}\varphi_a(f_n(0)),
\end{equation*}
and in consequence
\begin{equation*}
\varphi_a(x)=\sum_{n=0}^{N}\varphi_a(f_n(x))-\sum_{n=1}^{N}\varphi_a(f_n(0))
\end{equation*}
and
\begin{equation*}
\varphi_s(x)=\sum_{n=0}^{N}\varphi_s(f_n(x))-\sum_{n=1}^{N}\varphi_s(f_n(0))
\end{equation*}
for every $x\in[0,1]$.
\end{proof}

For all $k\in\mathbb N$ and $n_1,\ldots,n_k\in\{0,\ldots,N\}$ denote by $f_{n_1,\ldots,n_k}$ the composition $f_{n_1}\circ\cdots\circ f_{n_k}$.
We extend the notation to the case $k=0$ by letting $f_{n_1,\ldots, n_0}$ being the identity.

\begin{lemma}\label{lem23}
Let $(n_k)_{k\in\mathbb N}$ be a sequence of elements of $\{0,\ldots,N\}$. Then the sequence $(f_{n_1,\ldots,n_k}(0))_{k\in\mathbb N}$ is increasing and the sequence $(f_{n_1,\ldots,n_k}(1))_{k\in\mathbb N}$ is decreasing. Moreover,
\begin{equation*}
\lim_{k\to\infty}f_{n_1,\ldots,n_k}(y)=\lim_{k\to\infty}f_{n_1,\ldots,n_k}(z)
\end{equation*}
for all $y,z\in[0,1]$.
\end{lemma}

\begin{proof}
Fix a sequence $(n_k)_{k\in\mathbb N}$ of elements of $\{0,\ldots,N\}$ and an integer number $k\geq 2$. From~\eqref{0<1} we have
\begin{equation*}
0\leq f_{n_k}(0)< f_{n_k}(1)\leq 1
\end{equation*}
and by the strict monotonicity of $f_{n_1,\ldots,n_{k-1}}$ we conclude that
\begin{equation*}
f_{n_1,\ldots,n_{k-1}}(0)\leq f_{n_1,\ldots,n_{k}}(0)<f_{n_1,\ldots,n_{k}}(1)\leq f_{n_1,\ldots,n_{k-1}}(1).
\end{equation*}

To complete the proof it is enough to observe that for all $y,z\in[0,1]$ and $k\in\mathbb N$ we have
\begin{equation*}
|f_{n_1,\ldots,n_k}(y)-f_{n_1,\ldots,n_k}(z)|\leq f_{n_1,\ldots,n_k}(1)-f_{n_1,\ldots,n_k}(0)\leq c^k,
\end{equation*}
where $c\in(0,1)$ is the largest Lipschitz constant of the given contractions $f_0,\ldots,f_N$.
\end{proof}

\begin{lemma}\label{lem24}
For every $x\in[0,1]$ there exists a sequence $(x_k)_{k\in\mathbb N}$ of elements of $\{0,\ldots,N\}$ such that
\begin{equation}\label{x}
x=\lim_{k\to\infty}f_{x_1,\ldots,x_k}(0).
\end{equation}
\end{lemma}

\begin{proof}
Fix $x\in[0,1]$ and observe that according to Lemma~\ref{lem23} it is enough to show that there exists a sequence $(x_k)_{k\in\mathbb N}$ of elements of $\{0,\ldots,N\}$ such that
\begin{equation}\label{fff}
f_{x_1,\ldots,x_k}(0)\leq x\leq f_{x_1,\ldots,x_k}(1)
\end{equation}
for every $k\in\mathbb N$.

By~\eqref{0<1} there exists $x_1\in\{0,\ldots,N\}$ such that
\begin{equation*}
f_{x_1}(0)\leq x\leq f_{x_1}(1).
\end{equation*}
Thus, \eqref{fff} holds for $k=1$.

Fix $k\in\mathbb N$ and assume inductively that there exist $x_1,\ldots,x_k\in\{0,\ldots,N\}$ such that \eqref{fff} holds. Then
\begin{equation*}
0\leq f_{x_1,\ldots,x_k}^{-1}(x)\leq 1
\end{equation*}
and by~\eqref{0<1} there exists $x_{k+1}\in\{0,\ldots,N\}$ such that
\begin{equation*}
f_{x_{k+1}}(0)\leq f_{x_1,\ldots,x_k}^{-1}(x)\leq f_{x_{k+1}}(1).
\end{equation*}
Hence
\begin{equation*}
f_{x_1,\ldots,x_{k+1}}(0)\leq x\leq f_{x_1,\ldots,x_{k+1}}(1),
\end{equation*}
and the proof is complete.
\end{proof}


\section{General case}\label{General case}
Fix positive real numbers $p_0,\ldots,p_N$ such that
\begin{equation}\label{prob}
\sum_{n=0}^{N}p_n=1.
\end{equation}
Then there exists a unique Borel probability measure $\mu$ such that
\begin{equation}\label{inv}
\mu(A)=\sum_{n=0}^{N}p_n\mu(f_n^{-1}(A))
\end{equation}
for every Borel set $A\subset [0,1]$ (see \cite{H1981}; cf.\ \cite{F1997}).
From now on the letter $\mu$ will be reserved for the unique Borel probability measure satisfying~\eqref{inv} for every Borel set $A\subset [0,1]$.

\begin{lemma}\label{lem31}
The measure $\mu$ is continuous.
\end{lemma}

\begin{proof}
As a first step we want to show that
\begin{equation}\label{f_n(0)}
\mu\big(\{f_n(0)\}\big)=\mu\big(\{f_n(1)\}\big)=0
\end{equation}
for every $n\in\{0,\ldots,N\}$.

Applying~\eqref{inv} and using~\eqref{0<1}, we obtain
\begin{equation*}
\mu(\{0\})=\mu(\{f_0(0)\})=\sum_{n=0}^{N}p_n\mu(\{f_n^{-1}(f_0(0))\})=
p_0\mu(\{0\})+\sum_{n=1}^{N}p_n\mu(\emptyset),
\end{equation*}
By the fact that $p_0\in(0,1)$ we conclude that
\begin{equation*}
\mu(\{f_0(0)\})=\mu(\{0\})=0.
\end{equation*}
Similarly, applying \eqref{inv}, \eqref{0<1} and the fact that $p_N\in(0,1)$ we conclude that
\begin{equation*}
\mu(\{f_N(1)\})=\mu(\{1\})=0.
\end{equation*}
If $n\in\{1,\ldots,N\}$, then applying again \eqref{inv} and~\eqref{0<1}, we obtain
\begin{equation*}
\mu(\{f_{n-1}(1)\})=\mu(\{f_n(0)\})=p_{n-1}\mu(\{1\})+p_{n}\mu(\{0\})=0.
\end{equation*}

Our second step is to prove that
\begin{equation}\label{formula2}
\mu\big([f_{n_1,\ldots,n_k}(0),f_{n_1,\ldots,n_k}(1)]\big)=
\prod_{n=0}^{N}p_n^{\#\{i\in\{1,\ldots,k\}:n_i=n\}}
\end{equation} 	 	
for all $k\in\mathbb N\cup\{0\}$ and $n_1,\ldots,n_k\in\{0,\ldots,N\}$.

Since $\mu([0,1])=1$, it follows that \eqref{formula2} is satisfied for $k=0$.

Fix $k\in\mathbb N\cup\{0\}$ and assume that \eqref{formula2} holds for all $n_1,\ldots,n_{k}\in\{0,\ldots,N\}$. Fix also $n_{k+1}\in\{0,\ldots,N\}$.

Note first that from \eqref{f_n(0)}, \eqref{0<1}, and \eqref{inv}, we get
\begin{equation}\label{mu(B)}
\mu(B)=p_{n}\mu(f_n^{-1}(B))
\end{equation}
for all $n\in\{0,\ldots,N\}$ and Borel sets $B\subset[f_n(0),f_n(1)]$. This jointly with~\eqref{formula2} implies
\begin{equation*}
\begin{split}
\mu\big([f_{n_1,\ldots,n_{k+1}}(0),f_{n_1,\ldots,n_{k+1}}(1)]\big)
&=p_{n_1}\mu\big([f_{n_2,\ldots,n_{k+1}}(0),f_{n_2,\ldots,n_{k+1}}(1)]\big)\\
&=p_{n_1}\prod_{n=0}^{N}p_n^{\#\{i\in\{2,\ldots,k+1\}:n_i=n\}}\\
&=\prod_{n=0}^{N}p_n^{\#\{i\in\{1,\ldots,k+1\}:n_i=n\}}.
\end{split}
\end{equation*}

To prove that $\mu$ is continuous it is sufficient to show that $\mu$ has no atoms.

Fix $x\in[0,1]$. From Lemma~\ref{lem24} we conclude that there exists a sequence $(x_k)_{k\in\mathbb N}$ of elements of $\{0,\ldots,N\}$ such that
\eqref{x} holds. Then applying Lemma~\ref{lem23} and \eqref{formula2} with $n_i=x_i$ for $i\in \{1,\ldots,k\}$, we obtain
\begin{equation*}
\begin{split}
\mu(\{x\})&=\mu\left(\bigcap_{k\in\mathbb N} \big[f_{x_1,\ldots,x_k}(0), f_{x_1,\ldots,x_k}(1)\big]\right)\\
&=\lim_{k\to\infty}\mu\left(\big[f_{x_1,\ldots,x_k}(0), f_{x_1,\ldots,x_k}(1)\big]\right)\\
&=\lim_{k\to\infty}\prod_{n=0}^{N}p_n^{\#\{i\in\{1,\ldots,k\}:x_i=n\}}
\leq\lim_{k\to\infty}{\left(\max\{p_0,\ldots,p_N\}\right)}^k=0,
\end{split}
\end{equation*}
and the proof is complete.
\end{proof}

The next lemma is folklore (the reader can consult \cite{DF1996,PSS2000} in the case where $f_0,\ldots,f_N$ are similitudes and \cite{LaMy1994} in the case where $f_0,\ldots,f_N$ are contractions). More general results in this direction can be found e.g.\ in \cite{S2000a,S2000b}.

\begin{lemma}\label{lem32}
The measure $\mu$ is either singular or absolutely continuous with respect to the Lebesgue measure on $\mathbb R$.
\end{lemma}

Define the function $\varphi\colon[0,1]\to[0,1]$ by
\begin{equation*}
\varphi(x)=\mu([0,x]).
\end{equation*}

From now on the letter $\varphi$ will be reserved for the just defined function.

\begin{theorem}\label{thm32}
Either $\varphi\in\mathcal C_a$ or $\varphi\in\mathcal C_s$.
\end{theorem}

\begin{proof}
We first prove that $\varphi\in\mathcal C$.
	
That $\varphi$ is increasing is a consequence of the monotonicity of $\mu$.
The continuity of $\varphi$ and that $\varphi(0)=0$ follows from Lemma~\ref{lem31}. Since $\mu$ is a probability measure, we have $\varphi(1)=1$.

From~\eqref{mu(B)} we get
\begin{equation*}
\mu(f_n(B))=p_n\mu(B)
\end{equation*}
for all $n\in\{0,\ldots,N\}$ and Borel sets $B\subset[0,1]$. This jointly with~\eqref{prob} gives
\begin{equation*}
\sum_{n=0}^{N}\mu(f_n(B))=\sum_{n=0}^{N}p_n\mu(B)=\mu(B)
\end{equation*}
for every Borel set $B\subset[0,1]$.
Hence,
\begin{equation*}
\begin{split}
\varphi(x)&=\mu([0,x])=\sum_{n=0}^{N}\mu\big(f_n([0,x])\big)=\sum_{n=0}^{N}\mu\big([f_n(0),f_n(x)]\big)\\
&=\sum_{n=0}^{N}\mu\big([0,f_n(x)]\big)-\sum_{n=0}^{N}\mu\big([0,f_n(0)]\big)\\
&=\sum_{n=0}^{N}\varphi(f_n(x))-\sum_{n=0}^{N}\varphi(f_n(0))=
\sum_{n=0}^{N}\varphi(f_n(x))-\sum_{n=1}^{N}\varphi(f_n(0))
\end{split}
\end{equation*}
for every $x\in[0,1]$.

Thus, we have proved that $\varphi\in\mathcal C$. Now the assertion of the lemma follows from Lemma~\ref{lem32}; to see it the reader can consult \cite[Theorem 31.7]{B1995}.
\end{proof}

It is a very difficult (and still open) problem to decide for which parameters $p_0,\ldots,p_N$ the function $\varphi$ is absolutely continuous. However, it turns out that under some assumptions on the given contractions $f_0,\ldots,f_N$ equation~\eqref{e} has exactly one absolutely continuous solution in the class $\mathcal C$.

\begin{theorem}\label{thm33}
Assume that $f_0,\ldots,f_N\in C^2([0,1])$ and there exist $\lambda\in(0,1)$ and $c\in(0,\infty)$ such that $0<f_n'(x)\leq\lambda$ and $f_n''(x)\leq cf_n'(x)$ for all $n\in\{0,\ldots,N\}$ and $x\in[0,1]$. Then $\mathcal C_a$ consists of exactly one function.
\end{theorem}

\begin{proof}
Define $S\colon[0,1]\to[0,1]$ by
\begin{equation*}
S(x)=
\begin{cases}
f_n^{-1}(x)& \text{for $x\in\clopen{f_n(0)}{f_n(1)}$ and $n\in\{0,\ldots,N\}$},\\
1&\text{for $x=1$}.
\end{cases}
\end{equation*}
Now it is enough to apply \cite[Theorem 6.2.1]{LM1994}.
\end{proof}

Theorem~\ref{thm33} enforces looking for these unique parameters $p_0,\ldots,p_N$ for which $\varphi\in\mathcal{C}_a$. It is still difficult in full generality. However, it can be done with success in the case where $f_0,\ldots,f_N$ are similitudes; such a case will be considered in the next section.

Now let us set down an obvious characterization of these contractions $f_0,\ldots,f_N$ for which $\id_{[0,1]}\in\mathcal{C}_a$.

\begin{prop}\label{prop34}
The identity on $[0,1]$ belongs to $\mathcal{C}_a$ if and only if
\begin{equation}\label{w}
\sum_{n=0}^{N}f_n(x)-x=\sum_{n=1}^{N}f_n(0)
\end{equation}
for every $x\in[0,1]$.
\end{prop}

The last result of this section gives a precise formula for $\varphi$.

\begin{theorem}\label{thm35}
Assume that $x\in[0,1]$ and let $(x_k)_{k\in\mathbb N}$ be a sequence of elements of $\{0,\ldots,N\}$ such that \eqref{x} holds. Then
\begin{equation*}
\varphi(x)= \sum_{k=1}^{\infty}{\rm sgn}(x_k) \left[\prod_{n=0}^{N}p_n^{\#\{i\in\{1,\ldots,k-1\}:x_i=n\}}\cdot\sum_{n=0}^{x_k-1}p_n\right].
\end{equation*}
\end{theorem}

\begin{proof}
We begin with showing inductively that
\begin{equation}\label{formula1}
\hspace*{-2ex}
\mu\big([f_{n_1,\ldots,n_{k-1}}(0),f_{n_1,\ldots,n_{k}}(0)]\big)
={\rm sgn}(n_{k})\!\prod_{n=0}^{N}\!p_n^{\#\{i\in\{1,\ldots,k-1\}:n_i=n\}}\cdot\!
\sum_{n=0}^{n_{k}-1}\!p_n
\end{equation}	
for all $k\in\mathbb N$ all $n_1,\ldots,n_{k}\in\{0,\ldots,N\}$.

If $n_1=0$, then $\sign(n_1)=0$, and hence
\begin{equation*}
\mu\big([0,f_{n_1}(0)]\big)=\mu\big(\{0\}\big)=0={\rm sgn}(n_1)\sum_{n=0}^{n_1-1}p_n.
\end{equation*}
If $n_1\geq 1$, we have $\sign(n_1)=1$, and then by \eqref{0<1}, \eqref{mu(B)} and Lemma~\ref{lem31} we obtain
\begin{equation*}
\mu\big([0,f_{n_1}(0)]\big)=\sum_{n=0}^{n_1-1}\mu\big([f_n(0),f_n(1)]\big)
=\sum_{n=0}^{n_1-1}p_n\mu([0,1])={\rm sgn}(n_1)\sum_{n=0}^{n_1-1}p_n.
\end{equation*}
Therefore \eqref{formula1} holds for $k=1$ and all $n_1,\ldots,n_{k}\in\{0,\ldots,N\}$.
	
Fix $k\in\mathbb N$ and assume that \eqref{formula1} holds for all $n_1,\ldots,n_{k}\in\{0,\ldots,N\}$.
	
Fix $n_{k+1}\in\{0,\ldots,N\}$. Applying \eqref{mu(B)} and \eqref{formula1} we get
\begin{equation*}
\begin{split}
\mu\big([f_{n_1,\ldots,n_{k}}(0),f_{n_1,\ldots,n_{k+1}}(0)]\big)
&=\,p_{n_1}\mu\big([f_{n_2,\ldots,n_{k}}(0),f_{n_2,\ldots,n_{k+1}}(0)]\big)\\
&=p_{n_1}{\rm sgn}(n_{k+1})
\prod_{n=0}^{N}p_n^{\#\{i\in\{2,\ldots,k\}:n_i=n\}}\!\cdot\!\!\!\!
\sum_{n=0}^{n_{k+1}-1}p_n\\
&={\rm sgn}(n_{k+1})\prod_{n=0}^{N}p_n^{\#\{i\in\{1,2,\ldots,k\}:n_i=n\}}\!\cdot\!\!\!\!
\sum_{n=0}^{n_{k+1}-1}p_n.
\end{split}
\end{equation*}
	
By the continuity of $\varphi$ (see Theorem~\ref{thm32}) we have
\begin{equation*}
\varphi(x)=\varphi\left(\lim_{l\to\infty}f_{x_1,\ldots,x_l}(0)\right)
=\lim_{l\to\infty}\varphi(f_{x_1,\ldots,x_l}(0)).
\end{equation*}
Then using \eqref{0<1}, Lemma~\ref{lem31} and \eqref{formula1} with $n_i=x_i$ for all $i\in \{1,\ldots l\}$, we get
\begin{equation*}
\begin{split}
\varphi(f_{x_1,\ldots,x_l}(0))&=\mu([0,f_{x_1,\ldots,x_l}(0)])=
\sum_{k=1}^l\mu([f_{x_1,\ldots,x_{k-1}}(0),f_{x_1,\ldots,x_k}(0)])\\
&=\sum_{k=1}^l{\rm sgn}(x_{k})\left[\prod_{n=0}^{N}
p_n^{\#\{i\in\{1,\ldots,k-1\}:x_i=n\}}\cdot\sum_{n=0}^{{x_k}-1}p_n\right].
\end{split}
\end{equation*}
Passing with $l$ to $\infty$ we obtain the required formula for $\varphi$.
\end{proof}


\section{Similitudes case}\label{Similitudes case}
Throughout this section we assume that $f_0,\ldots,f_N$ are similitudes, i.e.\ there exist real numbers $\rho_0,\ldots,\rho_N\in(0,1)$ such that
\begin{equation}\label{rho}
\sum_{n=0}^N\rho_n=1
\end{equation}
and
\begin{equation*}
f_n(x)=\rho_n x+\sum_{k=0}^{n-1}\rho_k
\end{equation*}
for all $x\in[0,1]$ and $n\in\{0,\ldots,N\}$.

Note that \eqref{0<1} holds.

Since the above defined similitudes satisfy the assumptions of Theorem~\ref{thm33}, it follows that the class $\mathcal C$ has exactly one absolutely continuous solution. Thus according to Theorem~\ref{thm32} we conclude that $\varphi$ is singular except one very particular case of parameters $p_0,\ldots,p_N$, which we are looking for.

\begin{theorem}\label{thm41}
If $p_n=\rho_n$ for every $n\in\{0,\ldots,N\}$, then $\varphi=\id_{[0,1]}$.
\end{theorem}

\begin{proof}
Assume that $p_n=\rho_n$ for every $n\in\{0,\ldots,N\}$.	
	
Observe first that applying~\eqref{rho}, we get
\begin{equation*}
\sum_{n=0}^{N}f_n(x)-x=\sum_{n=0}^{N}\rho_n x+\sum_{n=0}^{N}\sum_{k=0}^{n-1}\rho_k-x=\sum_{n=0}^{N}f_n(0)=\sum_{n=1}^{N}f_n(0)
\end{equation*}
for every $x\in[0,1]$. Thus, $\id_{[0,1]}\in\mathcal C_a$, by Proposition~\ref{prop34}.
	
Now we can use Theorem~\ref{thm33} or argue as follows.
	
Denote by $\nu$ the one-dimensional Lebesgue measure restricted to $[0,1]$.
According to \cite[Theorem 12.4]{B1995} we infer that $\nu$ is the unique Borel measure on $[0,1]$ such that $\nu([0,x])=x$ for every $x\in[0,1]$.
Fix $n\in\{0,\ldots,N\}$ and choose $x\in\big[f_n(0),f_n(1)\big]$. Then
\begin{equation*}
\begin{split}
\nu([f_n(0),x])&=\nu([0,x])-\nu([0,f_n(0)])=x-f_n(0)=
\rho_n\left(\frac{x}{\rho_n}-\sum_{k=0}^{n-1}\frac{\rho_k}{\rho_n}\right)\\
&=p_n f_n^{-1}(x)=p_n\nu\big([0,f_n^{-1}(x)]\big)=
p_n\nu\big(f_n^{-1}([f_n(0),x])\big).
\end{split}
\end{equation*}
Hence
\begin{equation*}
\nu(A)=p_n\nu(f_n^{-1}(A))
\end{equation*}
for every Borel set $A\subset[f_n(0),f_n(1)]$, and in consequence,
\begin{equation*}
\nu(A)=\sum_{n=0}^N p_n\nu(f_n^{-1}(A))
\end{equation*}
for every Borel set $A\subset[0,1]$. Finally, by the uniqueness of $\mu$ we obtain
\begin{equation*}
\varphi(x)=\mu([0,x])=\nu([0,x])=x
\end{equation*}
for every $x\in[0,1]$.
\end{proof}

Combining Theorems~\ref{thm32}, \ref{thm33} and \ref{thm41} we get the following corollary.

\begin{corrolary}\label{cor43}
If  $p_n\neq \rho_n$ for some $n\in\{0,\ldots,N\}$, then $\varphi\in\mathcal C_s$.
\end{corrolary}

Note that in our setting $\prod_{n=0}^{N}p_n^{p_n}\rho_n^{-p_n}\geq 1$.
Observe also that the iterated function system consisting of the contractions $f_0,\ldots,f_N$ satisfies the open set condition.
Therefore Theorem~\ref{thm41} jointly with Corollary~\ref{cor43} can be written in the following form, which corresponds to Theorem 1.1 from \cite{NW2005}.

\begin{theorem}\label{thm42}
We have $\varphi\in\mathcal C_a$ if and only if $p_n=\rho_n$ for every $n\in\{0,\ldots,N\}$. Moreover, if $\varphi\in\mathcal C_a$, then $\varphi=\id_{[0,1]}$.
\end{theorem}

To the end of this section we assume that
\begin{equation*}
\rho_0=\rho_1=\cdots=\rho_N=\frac{1}{N+1}.
\end{equation*}

Note that \eqref{rho} is satisfied and equation~\eqref{e} now takes the form
\begin{equation*}\label{e1}
\varphi(x)=\sum_{n=0}^{N}\varphi\left(\frac{x+n}{N+1}\right)-
\sum_{n=1}^{N}\varphi\left(\frac{n}{N+1}\right)\tag{$\textsf{e}_N$}.
\end{equation*}
It is clear that for $N=1$ equation~\eqref{e1} reduces to equation~\eqref{e0}.

Fix $x\in[0,1]$ and define a sequence $(x_k)_{k\in\mathbb N}$ of elements of $\{0,\ldots,N\}$ as follows:

if $x=1$ we put $x_k=N$ for every $k\in\mathbb N$;

if $x<1$ we put $x_1=\left[(N+1)x\right]$ and then inductively
\begin{equation*}
x_{k+1}=\left[{(N+1)}^{k+1}x-\sum_{i=1}^{k}{(N+1)}^{k+1-i}x_i\right]
\end{equation*}
for every $k\in\mathbb N$, where $[y]$ denotes the integer part of $y\in\mathbb R$.

Clearly,
\begin{equation*}
x=\lim_{k\to\infty}f_{x_1,\ldots,x_k}(0)=\sum_{k=1}^\infty\frac{x_k}{{(N+1)}^k},
\end{equation*}
and Theorem~\ref{thm35} yields
\begin{equation}\label{varphi}
\varphi\left(\sum_{k=1}^\infty\frac{x_k}{{(N+1)}^k}\right)=
\sum_{k=1}^{\infty}{\rm sgn}(x_k) \!\left[\prod_{n=0}^{N}p_n^{\#\{i\in\{1,2,\ldots,k-1\}:x_i=n\}}\!\cdot\!\!\sum_{n=0}^{x_k-1}p_n\right]\!.
\end{equation}
In particular,
\begin{equation}\label{formula}
\varphi\left(\frac{n}{N+1}\right)=\sum_{k=0}^{n-1}p_k
\end{equation}
for every $n\in\{1,\ldots,N\}$.

Now we are able to calculate the integral of $\varphi$ on $[0,1]$.

\begin{prop}\label{prop44}
We have
\begin{equation*}
\int_0^1\varphi(x)dx=\frac{1}{N}\sum_{n=1}^{N}np_{N-n}.
\end{equation*}
\end{prop}

\begin{proof}
Using \eqref{formula} and~\eqref{e1}, we get
\begin{equation*}
\begin{split}
\int_0^1\varphi(x)dx&=
\sum_{n=0}^{N}\int_0^1\varphi\left(\frac{x+n}{N+1}\right)dx-
\int_0^1\sum_{n=1}^{N}\varphi\left(\frac{n}{N+1}\right)dx\\
&=(N+1)\sum_{n=0}^{N}\int_{\frac{n}{N+1}}^{\frac{n+1}{N+1}}\varphi\left(y\right)dy-\sum_{n=1}^{N}\varphi\left(\frac{n}{N+1}\right)\\
&=(N+1)\int_0^1\varphi(x)\,dx-\sum_{k=0}^{N-1}(N-k)p_k.
\end{split}
\end{equation*}
This implies the required formula for the integral of $\varphi$.
\end{proof}

We end this section observing that $\varphi$ can be extended to an increasing and continuous function satisfying~\eqref{e1} for every $x\in\mathbb R$.

\begin{prop}\label{prop45}
The function $\phi\colon\mathbb R\to\mathbb R$ given by
\begin{equation*}
\phi(x)=[x]+\varphi(x-[x])
\end{equation*}
is increasing, continuous and satisfies~\eqref{e1} for every $x\in\mathbb R$.	
\end{prop}

\begin{proof}
Fix $x\in\mathbb R$ and assume that $x\in\clopen{m(N+1)+l}{m(N+1)+l+1}$ for some $m\in\mathbb Z$ and $l\in\{0,1,\ldots,N\}$. Then $[\frac{x+i}{N+1}]=m$ for every $i\in\{0,\ldots,N-l\}$ and $[\frac{x+i}{N+1}]=m+1$ for every $i\in\{N-l+1,\ldots,N\}$.
Consequently,
\begin{equation*}
\begin{split}
\phi(x)&=[x]+\varphi(x-[x])=\\
&=m(N+1)+l+\sum_{n=0}^{N}\varphi\left(\frac{x-[x]+n}{N+1}\right)-
\sum_{n=1}^{N}\varphi\left(\frac{n}{N+1}\right)\\
&=m(N+1)+l+\sum_{n=0}^{N}\varphi\left(\frac{x+n-l}{N+1}-m\right)-
\sum_{n=1}^{N}\varphi\left(\frac{n}{N+1}\right)\\
&=(m+1)l+\sum_{n=0}^{l-1}\varphi\left(\frac{x+n-l+N+1}{N+1}-m-1\right)\\
&\quad+m(N+1-l)+\sum_{n=l}^{N}\varphi\left(\frac{x+n-l}{N+1}-m\right)-
\sum_{n=1}^{N}\varphi\left(\frac{n}{N+1}\right)\\
&=(m+1)l+\sum_{n=N+1-l}^{N}\varphi\left(\frac{x+n}{N+1}-m-1\right)\\
&\quad+m(N-l+1)+\sum_{n=0}^{N-l}\varphi\left(\frac{x+n}{N+1}-m\right)-
\sum_{n=1}^{N}\varphi\left(\frac{n}{N+1}\right)\\
&=\sum_{n=N-l+1}^{N}\left\{\left[\frac{x+n}{N+1}\right]+
\varphi\left(\frac{x+n}{N+1}-\left[\frac{x+n}{N+1}\right]\right)\right\}\\
&\quad+\sum_{n=0}^{N-l}\left\{\left[\frac{x+n}{N+1}\right]+
\varphi\left(\frac{x+n}{N+1}-\left[\frac{x+n}{N+1}\right]\right)\right\}-
\sum_{n=1}^{N}\varphi\left(\frac{n}{N+1}\right)\\
&=\sum_{n=0}^{N}\phi\left(\frac{x+n}{N+1}\right)-
\sum_{n=1}^{N}\phi\left(\frac{n}{N+1}\right).
\end{split}
\end{equation*}

To prove that $\phi$ is increasing fix $x<y$.
If $[x]=[y]$, then
\begin{equation*}
\phi(x)=[x]+\varphi(x-[x])=[y]+\varphi(x-[y])\leq [y]+\varphi(y-[y])=\phi(y),
\end{equation*}
and if $[x]<[y]$, then
\begin{equation*}
\phi(x)=[x]+\varphi(x-[x])\leq [y]\leq [y]+\varphi(y-[y])=\phi(y).
\end{equation*}

It is clear that $\phi$ is continuous at every point of the set  $\mathbb R\setminus\mathbb Z$. If $k\in\mathbb Z$, then by the continuity of $\varphi$ and~\eqref{cond} we obtain
\begin{equation*}
\lim_{x\to k^+}\phi(x)=\lim_{x\to k^+}\big([x]+\varphi(x-[x])\big)=k+\lim_{y\to 0^+}\varphi(y)=k=\phi(k)
\end{equation*}
and
\begin{equation*}
\lim_{x\to k^-}\phi(x)=\lim_{x\to k^-}\big([x]+\varphi(x-[x])\big)=k-1+
\lim_{y\to 1^-}\varphi(y)=k,
\end{equation*}
which completes the proof.
\end{proof}


\section{Matkowski-Weso\l owski case}\label{MW case}
First of all observe that formula~\eqref{varphi} with $N=1$ coincides with formula~\eqref{p}. So the main part of assertion~(\ref{representation}) of Theorem~\ref{thm11} is a very special case of Theorem~\ref{thm35}, whereas its moreover part follows from Corollary \ref{cor43}.
Now we would like to get a little bit more information about the class ${\mathcal C}$. For this purpose, we denote the convex hull of a set $A$ by ${\rm conv}(A)$  and put
\begin{equation*}
{\mathcal W}=\{\varphi_p:p\in(0,1)\},
\end{equation*}
where $\varphi_p\colon[0,1]\to[0,1]$ is the function defined by~\eqref{p}.

\begin{prop}\label{prop51}
The set $\mathcal W$ is linearly independent.
Moreover:
\begin{enumerate}
\item\label{convexSet} ${\rm conv}({\mathcal W})\subset\mathcal C$;
\item\label{convexAndSingular} ${\rm conv}({\mathcal W}\setminus\{\varphi_{\frac{1}{2}}\})\subset\mathcal C_s$.
\end{enumerate}
\end{prop}

\begin{proof}
To prove that $\mathcal W$ is linearly independent fix $n\in\mathbb N$, $\alpha_1,\ldots,\alpha_n\in\mathbb R$, $0<p_1<p_2<\dots<p_n<1$ and assume that
\begin{equation*}
\sum_{i=1}^n\alpha_i\varphi_{p_i}(x)=0.
\end{equation*}
for every $x\in[0,1]$. Applying~\eqref{p} we conclude that $\varphi_{p_i}(\frac{1}{2^k})=p_i^k$ for all $k\in\mathbb N$ and $i\in\{1,\ldots,n\}$. Then for every $k\in\mathbb N$ we have
\begin{equation*}
\sum_{i=1}^{n}\alpha_i\left(\frac{p_i}{p_n}\right)^k=0.
\end{equation*}
Taking the limit as $k\to \infty$ we get $\alpha_n=0$. Repeating this procedure $n-1$ times gives $\alpha_n=\alpha_{n-1}=\dots=\alpha_1=0$.

Assertion~(\ref{convexSet}) follows from Remark~\ref{rem21} and assertion~(\ref{convexAndSingular}) is a consequence of the moreover part of assertion~(\ref{convexAndSingular}) of Theorem~\ref{thm11}.
\end{proof}

To formulate an answer to the problem posed in \cite{M1985} by Janusz Matkowski
define first a function $\varphi_1\colon[0,1]\to\mathbb R$ putting $\varphi_1(x)=1$ and observe that by Proposition~\ref{prop51} and the fact that $\varphi_1(0)=1$ and $\varphi_p(0)=0$ for every $p\in(0,1)$ the set
${\mathcal W}\cup\{\varphi_1\}$ is linearly independent. Let $\mathcal M$ denote the vector space whose basis is ${\mathcal W}\cup\{\varphi_1\}$, i.e.
\begin{equation*}
{\mathcal M}={\rm lin}\big({\mathcal W}\cup\{\varphi_1\}\big).
\end{equation*}
Applying Proposition~\ref{prop51} and Remark~\ref{rem12}, we get the following result.

\begin{theorem}\label{thm52}
Every function belonging to $\mathcal M$ is a continuous solution of equation~\eqref{e0}.
Moreover, $\sum_{i=1}^n\alpha_i\varphi_{p_i}\in\mathcal M$ is:
\begin{enumerate}
\item monotone provided that ${\rm sgn}(\alpha_i)={\rm sgn}(\alpha_j)$ for all $i,j\in\{1,\ldots,n\}$ such that $p_i,p_j\in(0,1)$;
\item singular for all $p_1,\ldots,p_n\in (0,\frac{1}{2})\cup(\frac{1}{2},1]$.
\end{enumerate}
\end{theorem}
\bigskip

{\bf Acknowledgement.} The research of the first author was supported by the Silesian University Mathematics Department (Iterative Functional Equations and Real Analysis program). Furthermore, the research leading to these results has received funding from the European Research  Council under the European Union's Seventh Framework Programme (FP/2007-2013) / ERC Grant Agreement n.291497.

\end{document}